\documentclass[12pt]{article}

\usepackage{graphics}
\usepackage{latexsym}
\usepackage{color}
\usepackage{epsfig}

\usepackage{amsthm}
\usepackage{graphicx}
\usepackage{color}
\usepackage{amsmath}
\usepackage{amssymb}
\usepackage{amscd}
\usepackage{latexsym}
\usepackage{graphics}
\usepackage{verbatim}
\usepackage{hyperref}

\theoremstyle{plain}
\newtheorem{theorem}{Theorem}[section]
\newtheorem{lemma}[theorem]{Lemma}
\newtheorem{corollary}[theorem]{Corollary}

\newtheorem{observation}[theorem]{Observation}
\newtheorem{example}[theorem]{Example}

\theoremstyle{definition}
\newtheorem{definition}[theorem]{Definition}

\theoremstyle{remark}

\newcommand{\ignore}[1]{}

\begin{document}

\title{Isoperimetric Sequences for Infinite Complete Binary Trees,
 Meta-Fibonacci  Sequences and
  Signed Almost Binary Partitions}

\author{
% \footnotemark[2] \\
L. Sunil Chandran\footnotemark[1] \and Anita Das\footnotemark[1]
\and Frank Ruskey \footnotemark[2] }

\date{}

\maketitle                 % Produces the title.
\renewcommand{\thefootnote}{\fnsymbol{footnote}}
% \footnotetext[2]{Somewhere, someplace}
\footnotetext[1]{ Computer Science and Automation Department, Indian
Institute of Science, Bangalore - 560012, India. {\tt\{sunil,anita\}@csa.iisc.ernet.in.} The second author's research is supported by Dr. D.S. Kothari fellowship from University Grants Commission.}

\footnotetext[2]{ Department of Computer Science, University of
Victoria, Victoria, BC, Canada. {\tt ruskey@cs.uvic.ca}. Research
supported in part by NSERC.}
\renewcommand{\thefootnote}{\arabic{footnote}}

\begin{abstract}

In this paper we demonstrate connections between three seemingly
unrelated concepts.
\begin {enumerate}
\item  The discrete isoperimetric problem in the infinite  binary
       tree with all the leaves at the same level, $ {\mathcal T}_{\infty}$:
       % With respect to $ {\mathcal T}_{\infty}$,
       The $n$-th
       edge isoperimetric number
       $\delta(n)$ is defined to be $\min_{|S|=n, S \subset V( {\mathcal T}_{\infty}) } |(S,\overline S)|$,
       where $(S,\overline S)$ is the set of edges in the cut
       defined by $S$.

\item  Signed almost binary partitions: This is the special
       case of the coin-changing problem where the coins are drawn from
       the set  $\{ \pm (2^d - 1):  \mbox { $d$ is a positive integer } \}$.
       The quantity of interest is $\tau(n)$, the minimum number of coins necessary to make change for
       $n$ cents.

\item  Certain Meta-Fibonacci sequences:
       % These are fascinating concepts, originating
       % from the phylosophical work of Hofstadter, ``G\"odel, Escher, Bach.
       % An Eternal Golden Braid''. In this paper we consider two of the
       The Tanny sequence is defined by $T(n)=T(n{-}1{-}T(n{-}1))+T(n{-}2{-}T(n{-}2))$
       and the Conolly sequence is defined by $C(n)=C(n{-}C(n{-}1))+C(n{-}1{-}C(n{-}2))$,
       where the initial conditions are $T(1) = C(1) = T(2) = C(2) = 1$.
       These are well-known ``meta-Fibonacci" sequences.

\end  {enumerate}

The main result that ties these three together is the following:
$$ \delta(n) = \tau(n) = n+ 2 + 2  \min_{1 \le k \le n} (C(k) - T(n-k) - k).$$
Apart from this,  we prove several other results which bring out the interconnections
between the above three concepts.

\end{abstract}


\section{Introduction and Background}

In this paper we consider three well-studied, but seemingly unrelated
concepts and bring out the interconnections between them. We begin
by describing each concept, together with some of its background.

\subsection {Discrete Isoperimetric Problem on Infinite Binary Trees}

Let $G = (V, E)$ be a graph. For $X \subseteq V(G)$, a cut $(X,
\overline{X})$ in $G$ is defined as the set $\{(u,v) \in E(G) | u \in X, \ v
\in V-X\}$. The $n$-th edge isoperimetric number of a graph $G$,
denoted $\delta(n,G)$ is the least number of edges in any cut
$(X,\overline{X})$ where $|X| = n$.  For finite graphs, we take
$1 \le n \le |V(G)|$. In the case of infinite graphs,
$\delta(1,G),\delta(2,G), \ldots,$ forms an infinite sequence.

The discrete isoperimetric problems  form a very useful and important
subject in graph theory and combinatorics. See \cite{Bol86}, Chapter
16 for a  brief introduction on isoperimetric problems. For a
detailed treatment see the book by Harper \cite{Har04}. See also the
surveys by Leader \cite {Lea91} and by Bezrukov \cite{Bez94,Bez99}
for a comprehensive overview of work in the area.
Isoperimetric problems are
typically studied for graphs with special (usually symmetric)
structure.
The study of
isoperimetric properties of binary trees was initiated by Otachi
{\sl et al.} \cite{otachi} and continued in
\cite{BharadwajChandran,BCD}.

\begin{figure}
\includegraphics[width=5in]{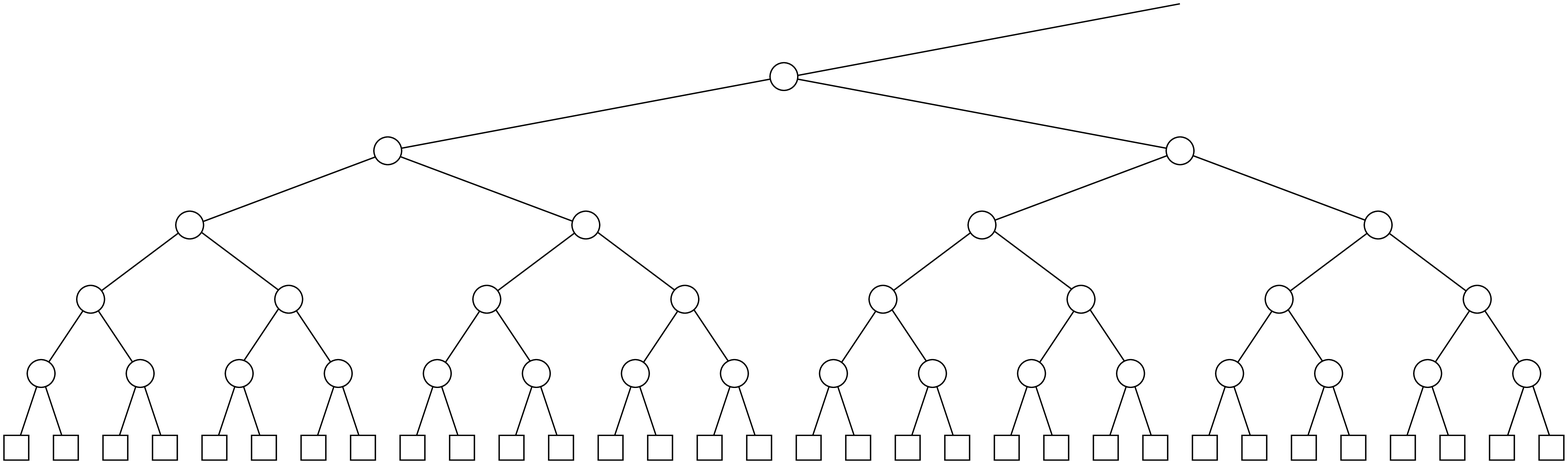}
\caption{The leftmost part of the infinite binary tree $\mathcal{T}_\infty$
with all leaves at the same level.}
\label{fig:BT1}
\end{figure}

Define the infinite binary tree $\mathcal{T}_\infty$ whose leaves
are all at the same level, as shown in Figure \ref{fig:BT1}.
In this paper we will study  the edge isoperimetric sequence
of  $\mathcal{T}_\infty$.  We will use  $\delta(n)$ to
denote $\delta(n, \mathcal {T}_{\infty})$.
 A typical cut in $\mathcal{T}_\infty$ is illustrated in Figure \ref{fig:BT3}.

We will also study  two natural variations of the edge
isoperimetric problem
on $\mathcal {T}_{\infty}$. The first one is by  restricting  $X$ to be connected
 i.e., we minimize over subsets $X$ of $V(\mathcal {T}_{\infty})$, where $X$ induces a subtree and
$|X| = n$. Then the minimum
value is called the $n$-th connected edge isoperimetric number
and is  denoted by  $\delta_C(n)$.
 In Figure \ref{fig:BT4}, on the right we have illustrated a subset $X$  of
 vertices with $|X|= 24$, inducing a subtree in ${\mathcal T}_{\infty}$,
 such that $|(X,\overline X)| = \delta_C(24)
= 2$.

The second variation  is by requiring that  the infinite set
$\overline{X}$ be  connected.  It is easy to see that  this
condition is equivalent to restricting  $X$  to induce
 a disjoint collection of
complete  binary trees with all leaves at the lowest level of
$\mathcal{T}_\infty$. In this case the minimum value is called the
$n$-th co-connected isoperimetric number  and is  denoted by
$\delta_P(n)$. \footnote { P in $\delta_P$ stands for `positive'. It 
is chosen to be consistent with the notation $\tau_P$ from section \ref {ABPsection}. }

 In Figure \ref{fig:BT4}, on the left we have illustrated a subset $X$  of
 vertices with $|X|= 24$, with $\overline X$
 inducing a subtree in ${\mathcal T}_{\infty}$,
 such that $|(X,\overline X)| = \delta_P(24)
= 4$. Note that $X$ consists of a collection of complete binary trees
with all leaves at the lowest level of ${\mathcal T}_{\infty}$.

\begin{figure}
\begin{center}
\includegraphics[width=5in]{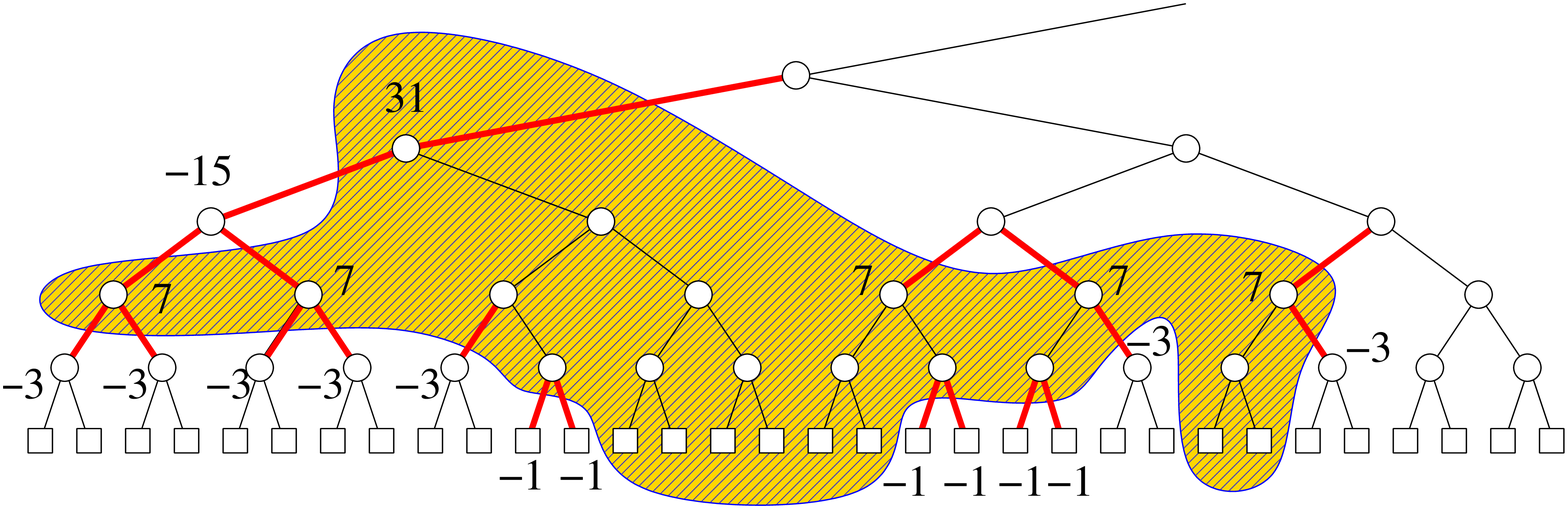}
\end{center}
\caption{A subset $S$ of $\mathcal{T}_\infty$ with $|S|=24$ and
$|(S,\overline{S})|=20$.  The numbers on the nodes are $f_S(v)$ from
section \ref {deltaTausection}.}
\label{fig:BT3}
\end{figure}

\begin{figure}
\includegraphics[width=5in]{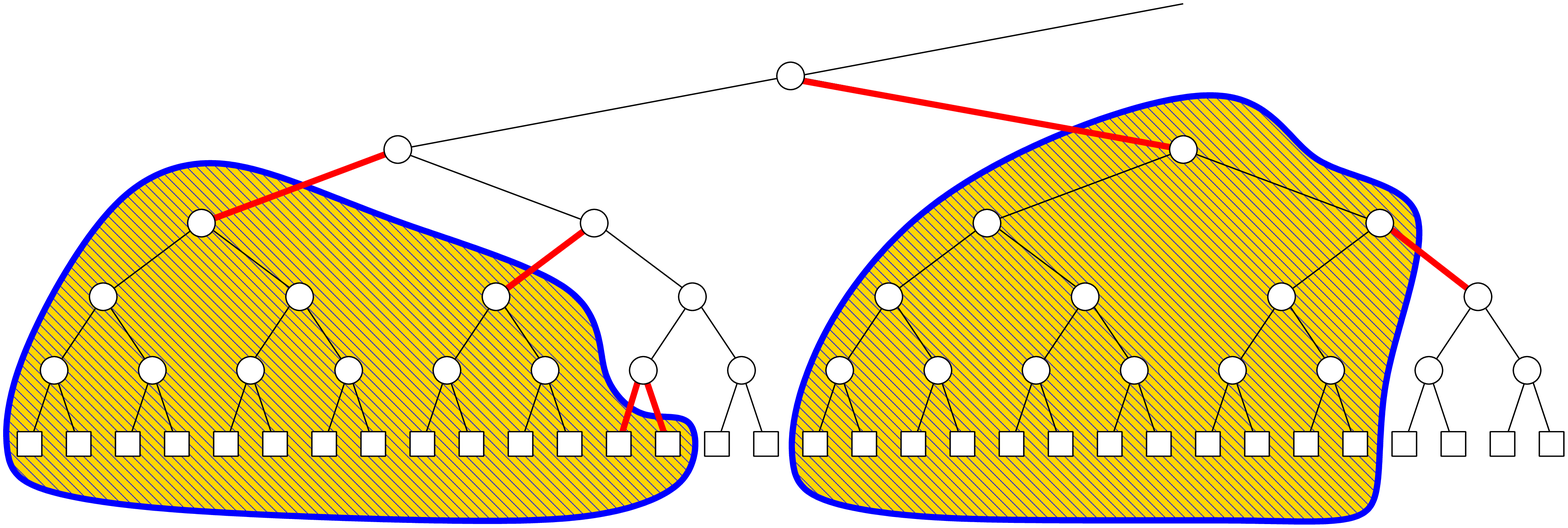}
\caption{On the left: a subset illustrating $\delta_P(24) =
\delta_P(15+7+1+1)= 4$. On the right: a subset illustrating
$\delta_C(24) = \delta_C(31-7) = 2$.} \label{fig:BT4}
\end{figure}

 {
\subsection{Almost binary partitions: A special case of coin changing problem}

\label {ABPsection}

}

We can state  the well-known coin changing problem
as follows: Let $F$ be a subset of integers, i.e.  $F \subseteq \mathcal Z$.  Given a positive integer $n$, find
the smallest $k$ such that $n$ can be partitioned in to $k$ parts, such
that each part belongs to $F$. In other words, we require a partition of $n$,
of the form $n = \sum_{1 \le i \le k} a_i$, where $a_i \in F$, for the
smallest possible $k$.
Note that here we do not assume that $a_i \ne a_j$ for $i \ne j$.
A \emph{binary partition} of a number $n$ is one that has all parts
of the form $2^k$, i.e. $F = \{ 2^k : \mbox { $k$ is a non-negative integer } \}$.
 Several papers have been written about binary partitions of integers,
e.g. Booth \cite{Booth}, Prodinger \cite{Prodinger} and Sawada
\cite{Sawada}.

We call a partition of $n$ of the  form  $\sum_{1 \le i \le k} a_i$
an `almost binary partition' (ABP)
if each $a_i \in  \{ 2^d - 1 : \mbox {  $d$ is a positive integer } \}$,
and a signed almost binary partition (SABP)  if each
$a_i \in \{ \pm (2^d - 1): \mbox {  $d$ is a positive integer } \}$.

The number $2^\ell-1$ occurs so often in the rest of this
paper that we adopt the following two notations for it: $\nu_\ell =
2^\ell-1$ or $\nu(\ell) = 2^\ell-1$. Furthermore we extend the
notation to sets, so that if $P$ is a multi-set of natural numbers,
then

\[ \nu(P) = \sum_{i \in P} \nu_i = \sum_{i \in P} (2^i-1).
\]

Note that
a SABP of $n$ is specified by two multisets $P$ (for positive) and
$N$ (for negative) such that

\[
n \ =\  \nu(P) - \nu(N) \ =\  \sum_{i \in P} (2^i-1) - \sum_{i \in N} (2^i-1).
\]

Sometimes we refer to the pair $(P,N)$ as the partition. We also use
the notation $|(P,N)|$ to mean $|P|+|N|$.
Note that an ABP of $n$ can be thought of as an SABP,  $(P,N)$ of $n$ where
$N = \emptyset$.

We also define the  connected SABP
(abbreviated as CABP) of $n$ to be a SABP $(P,N)$ of $n$, where $|P| = 1$.
The definition of CABP may look
somewhat unnatural, but it helps crucially  in establishing the
interconnections among  the three
problems studied in this paper.

 Define $\tau(n)$ to be the
least number of parts in any SABP of $n$. Similarly define
$\tau_C(n)$ and $\tau_P(n)$ to be the least number of parts in any
CABP and ABP of $n$, respectively. (The $P$ in notation $\tau_P$ stands for
positive, since all terms are required to be positive in an
ABP).  If a SABP (ABP or CABP)  has the least number
of parts then we will say that it is \emph{minimal}; it is one that
minimizes $|(P,N)| = |P|+|N|$.

\subsection{Meta-Fibonacci sequences}

In this paper we will study two of the most well-studied Meta-Fibonacci
sequences: The Tanny sequence, defined by S. Tanny  \cite {Tanny}  and the Conolly
sequence defined by B. W. Conolly \cite {Conolly}.
The \emph{Tanny sequence} is
given by the following recurrence relation, where $T(1)=T(2)=1$.

\begin{equation}
T(n) = T(n-1-T(n-1)) + T(n-2-T(n-2)), \ \ \ \ n > 2
\end{equation}

The \emph{Conolly sequence} is given by the following recurrence
relation, where $C(1)= C(2) = 1$.

\begin{equation}
C(n) = C(n-C(n-1)) + C( n-1-C(n-2)), \ \ \ \ n > 2
\label{eq:gfC}
\end{equation}

In \cite{JacksonRuskey} it is proven that the ordinary generating
functions $T(z)$ and $C(z)$ of the Tanny and Conolly numbers are

\begin{equation}
T(z) = z \sum_{n \ge 0} \prod_{k=1}^n (z+z^{2^k}) \;\;\; \text{ and
} \;\;\; C(z) = \frac{z}{1-z} \prod_{n \ge 0} (1+z^{2^n-1})
\label{eq:gfTC}
\end{equation}

\subsection {Our Results}

In this paper we prove several results which bring out the interconnections
among the three problems
described in the previous sections. The main result is the
following:

$$  \delta(n) = \tau (n) = n+ 2 +  2 \min_{0 \le k \le n}  (C(k) - T(n-k) - k)$$

The following  result which was derived as an intermediate step in proving
the main result, is of independent interest. This result allows to
prove a conjecture of J. Arndt,  from OEIS \cite {OEIS}, regarding the
generating function of the sequence, $\delta_P(n), n \in \mathcal N \setminus \{0\} $.

$$ \delta_P(n) = 2 C(n) - n$$

For all $n \ge 1$, it is clear that
$\delta(n) \le \delta_C(n)$ and $\delta(n) \le \delta_P(n)$. See
Table \ref{table:numbers} for the values of these sequences for
small values of $n$, along with the corresponding values of
$T(n)$ and $C(n)$.  In the OEIS, these are sequences
A005811, A100661, A192099, A006949 and A046699, respectively \cite{OEIS}.

\begin{table}
{\small
\begin{tabular}{c|ccccccccccccccccccccccc}
$n$      & 1 & 2 & 3 & 4 & 5 & 6 & 7 & 8 & 9 & 10 & 11 & 12 & 13 & 14 & 15 & 16 & 17 & 18 & 19 & 20 \\ \hline
$\delta_C(n)$ & 1 & 2 & 1 & 2 & 3 & 2 & 1 & 2 & 3 & 4 & 3 & 2 & 3 & 2 & 1 & 2 & 3 & 4 & 5 & 4 \\
$\delta_P(n)$ & 1 & 2 & 1 & 2 & 3 & 2 & 1 & 2 & 3 & 2 & 3 & 4 & 3 & 2 & 1 & 2 & 3 & 2 & 3 & 4 \\
$\delta(n)$   & 1 & 2 & 1 & 2 & 3 & 2 & 1 & 2 & 3 & 2 & 3 & 2 & 3 & 2 & 1 & 2 & 3 & 2 & 3 & 4 \\ \hline
$T(n)$   & 1 & 1 & 2 & 2 & 2 & 3 & 4 & 4 & 4 &  4 &  5 &  6 &  6 &  7 &  8 &  8 &  8 &  8 &  8 &  9 \\
$C(n)$   & 1 & 2 & 2 & 3 & 4 & 4 & 4 & 5 & 6 & 6 & 7 & 8 & 8 & 8 & 8 & 9 & 10 & 10 & 11 & 12 \\
\end{tabular}
} \caption{The values of $\delta(n),\delta_P(n)$, $\delta_C(n)$, $T(n)$ and $C(n)$ for $1 \le n \le 20$.}
\label{table:numbers}
\end{table}

In Table \ref {table:numbers}, it is remarkable how often the three values
$\delta_C(n)$, $\delta_P(n)$, and $\delta(n)$ are identical. The
first value of $n$ for which $\delta(n)$ is strictly less than both
$\delta_C(n)$ and $\delta_P(n)$ is when $n = 43$; then $\delta(43) =
3$ and $\delta_C(43) = \delta_P(43) = 5$. The first such even value
is $n = 282$. However, the number of times that $\delta_C(n) \neq
\delta_P(n)$ for $1 \le n \le 10^4$ is $7187$, so the true behavior
is only becoming apparent when $n$ is large.

In tune with the literature on discrete isoperimetric problems, the
most important question here is to find an explicit formula for
$\delta(n)$ in terms of $n$. But as in the case of many other 
graph classes,  
this looks extremely difficult at this stage. So it makes sense to
seek a better understanding of $\delta$ in terms of the easier
sequences $\delta_P$ and $\delta_C$.  (We will show in this paper
that these latter sequences are much easier to deal with than $\delta$:
For example, $\delta_P(n)$ and $\delta_C(n)$ can be computed in
$O(\log n)$ time, whereas as of now, we have only an $O(n)$ time algorithm
to compute $\delta(n)$.) In this context, the following questions
become relevant: What would be the necessary and sufficient conditions for
a number $n$ to satisfy the equality $\delta(n) = \delta_P(n)$ or 
$\delta(n) = \delta_C(n)$? Let ${\cal X} = \{ n \in \mathcal N \setminus \{0\}:
\delta(n) = \delta_P(n) \}$  and
${\cal Y} = \{ n \in \mathcal N \setminus \{0\}: \delta(n) = \delta_C(n) \}$.
Also let ${\cal X}_t = \{ n \in {\cal X}: n < t \}$ and
${\cal Y}_t = \{ n \in {\cal Y}: n < t \}$.  We show that there is a 
one to one correspondence between ${\cal X}_{2^k}$ and ${\cal Y}_{2^k}$,
for $k \ge 2$. It also follows that if we know the numbers in
${\cal X}_{2^k}$ then we can also get the numbers in ${\cal Y}_{2^k}$.
It follows that it is sufficient to study one of these two sets.

We are still unable to characterize the numbers that belong to ${\cal X}$,
but we give a non-trivial sufficient condition for a number $n$ to
belong to ${\cal X}$, in terms of the nature of the optimal ABP of $n$.
Suppose $n$ has
an ABP,  $ \mu_{i_1} + \mu_{i_2}
+ \ldots \mu_{i_t}$ with  $ i_1 > i_2 >  \ldots > i_t$
and $i_j   \ge i_{j+1} + k$ for $1 \le j \le t-1$, then
we say that this ABP of $n$ satisfies the gap $k$ condition.
We prove that $\delta(n) = \delta_P(n)$, if $n$ has
an ABP satisfying the gap-$3$ condition. It is  not possible to replace
the gap-$3$ condition with gap-$2$ condition:
there exist  numbers $n$  which satisfy the gap 2 condition, but
with $\delta(n) \ne \delta_P(n)$.

The gap-3 theorem discussed in the previous paragraph turned out to
have an unexpected consequence: We could improve the previously best
known lower bound on the edge isoperimetric peak of $B_d$, the 
complete binary tree of depth $d$, studied in 
\cite {otachi,BharadwajChandran,BCD}.

\section {Preliminaries on  (signed) almost binary partitions}

\subsection {Almost Binary Partitions}

Recall that  $\tau_P(n)$ is  the least number of parts possible
in an ABP of $n$.
For example $\tau_P(12) = 4$ since $12 = 3 + 3 + 3 +3 = 7 +
3 + 1 + 1$, and there is no way to write 12 using fewer parts of the
right form.  As mentioned before, this is an instance of a ``coin-changing problem" (make
change using the least number of coins), where the denominations of
coins are taken from the set $\mathbf{A} = \{
1,3,7,\ldots,2^k-1,\ldots \}$.
A \emph{greedy} solution to the coin changing problem is one where
the largest possible coin is successively chosen. For our earlier
example, the partition $7+3+1+1$ would be  the one chosen by the greedy
algorithm.
We define $\text{Greedy}(n)$ to be the multi-set of exponents that
are used in finding the greedy partition of $n$. For example
$\text{Greedy}(12) = \{3,2,1,1\}$, since $7 = 2^3-1,3=2^2-1,$ and
$1=2^1-1$.

We
will show that greedy algorithm outputs the
least number of coins if the denominations of coins
are from the set  $\mathbf {A} = \{ 2^d - 1: d  \mbox { is a positive integer } \}$.
Let $G_\infty$ be the greedy algoithm for the coin changing problem,
when the denominations come from the set $A  = \{ 2^d - 1: d  \mbox { is a positive integer } \}$ and let  $G_\infty(n)= |Greedy(n)|$ be
the number of parts in the partition of $n$
returned by the algorithm $G_\infty$.  Also
for $k \in \mathcal N \setminus \{ 0 \}$, let $G_k$ denote the greedy algorithm when
the denominations belong to
the set $\{ 2^d - 1:  d \le k, \mbox { and $d $ is a positive integer } \}$ and
let $G_k(n)$ be the number of parts in the partition on $n$, returned
by the algorithm $G_k$.

\begin{lemma}
The greedy algorithm solves the almost binary partition  problem.
In other words, $\tau_P(n) = |\text {Greedy} (n)|$.
\label{lemma:greedy}
\end{lemma}

\begin{proof}
According to a result of Magazine, Nemhauser, and Trotter \cite{MNT} (also described
in the book of Hu and Shing \cite{HuShing}), given that the greedy algorithm ${G}_k$
gives optimal solutions,
 the greedy algorithm
$G_{k+1}$ gives optimal solutions
if and only if there exist  $p_k$ and $\rho_k$ such that 
\[
1 + G_k(\rho_k) \le p_k, \mbox{ where } 2^{k+1}-1 = p_k (2^k-1) - \rho_k \mbox{ with } 0 \le \rho_k < 2^k-1.
\]
Solving the ``where" condition, we get $p_k = 3$ and $\rho_k = 2(2^{k-1}-1)$.
The greedy ABP for $\rho_k =  2(2^{k-1}-1)$
 is $(2^{k-1}-1)+(2^{k-1}-1)$ and thus $G_k(\rho_k) = 2$.
Thus the inequality $1 + G_k(\rho_k) \le p_k$ is satisfied for all $k$.
Clearly  greedy algorithms $G_1,G_2$ etc
give the optimum solution.  The result follows by induction.
\end{proof}

The following lemma implies that in the greedy solution there are at
most two equal values. Furthermore, if there are two equal values,
then they are the two smallest values.

\begin{lemma}
Let $d_1 \ge d_2 \ge \cdots \ge d_s$ be a sequence of positive
integers such that $\sum_{1 \le i \le s} \nu(d_i) = n$. Then

\begin {enumerate}

\item   $\{d_1,d_2,\ldots,d_s\} = \text{Greedy}(n)$ if and only if $d_1 >
d_2 > \ldots > d_{s-1}$.

\item   $\max \text { Greedy(n) }  = \left \lfloor \log (n+1)  \right \rfloor$.

\item   If  $\max \text {Greedy} (n) = d_1$, then $n  \le  2^{d_1 + 1} - 2$.

\end {enumerate}

 \label{lemma:greedseq}

\end{lemma}

\begin{proof}
We will show that
\begin{equation}
2^{\lfloor \log (n+1) \rfloor} -1 \le n \le 2 ( 2^{\lfloor \log (n+1)
\rfloor} -1 ). \label{ineq:greed}
\end{equation}

Note that

\[
n+2 \le 2^{\lceil \log(n+2) \rceil} \le 2^{\lfloor \log(n+2) \rfloor + 1}.
\]

But, unless $n+2 = 2^k$, we have $\lfloor \log(n+2) \rfloor = \lfloor
\log(n+1) \rfloor$. Thus, if $n+2 \ne 2^k$,  $n+2 \le 2^{\lfloor \log(n+1)
\rfloor + 1}$ which implies the right inequality in
(\ref{ineq:greed}). On the other hand, if $n+2 = 2^k$, then an easy
calculation shows that the right inequality is, in fact, an
equality.

These inequalities in  (\ref {ineq:greed})
   show that the integer first chosen by the greedy
algorithm is $2^{\lfloor \log  (n+1) \rfloor} -1$ and 
therefore $\max (Greedy(n)) = \lfloor \log (n+1) \rfloor$. From 
this, part (2) of the Lemma  immediately follows, and also part (3) follows from
the second inequality in (\ref {ineq:greed}).
  We also see from inequality (\ref {ineq:greed})
that $2^{\lfloor
\log (n+1) \rfloor} -1$ can be chosen at most twice. Furthermore, if
it is chosen twice, then the algorithm terminates. Now to formally
 prove part  (1) of Lemma, we can observe that $Greedy(n) = \{   \nu_{\lfloor 
\log (n+1) \rfloor}   \} \cup Greedy (n -  \nu_{\lfloor \log (n+1) \rfloor})$
and apply induction.

\end{proof}

\subsection{Signed almost binary partitions}

Let $(P,N)$ be a SABP of $n$. If $(P,N)$ is a minimal SABP then by
Lemma \ref{lemma:greedy} we may assume that $P =
\text{Greedy}(\nu(P))$ and $N = \text{Greedy}(\nu(N))$.

We say that a SABP is in \emph{normal form} if the following three conditions are met:
\begin{itemize}
\item[(A)]
$P \cap N = \emptyset$.
\item[(B)]
$P = \text{Greedy}(\nu(P))$ and $N = \text{Greedy}(\nu(N))$.
\item[(C)]
$\max(P) \in \{ \lfloor \log n \rfloor, \ 1+ \lfloor \log n \rfloor \}$.
\end{itemize}

\begin{theorem}
\label {thm:normalform}
Every positive integer $n$
has a minimal  SABP in normal form.
\label{thm:normal}
\end{theorem}

\begin{proof}
Let $d = \lfloor \log n \rfloor$ and let $(P,N)$ be an SABP of $n$.
We first claim that if   $(P,N)$ satisfies (A) and (B) and
if $\max(P) = d+1+c$ for some $c > 0$, then $d+c \in N$.
To see this first note that $\nu (P) > \nu (N)$
and since $P = \text Greedy(\nu(P))$ and $N= \text Greedy (\nu (N))$,
$\max (P) \ge \max (N)$.  Now if $d+c \notin N$,
 then in view of $(A)$ we can infer that $\max(N) \le d+c-1$.
By Lemma \ref {lemma:greedseq} (3), $\nu (N) \le 2^{d+c} - 2$. So

\begin{eqnarray}
n
&   = & (2^{d+c+1} -1) + \sum_{j \in P \setminus \{d+c+1\}} \nu_j - \nu (N)  \\
&   \ge  & (2^{d+c+1} -1) - 2^{d+c} + 2 > 2^{d+c} \ge 2^{ \lfloor \log n \rfloor + 1}
\end{eqnarray}

which is impossible.

We define the following two operations which operate on a SABP of $n$
and transform it into another SABP of $n$.

Operation 1:  Replace  $P$ by  $\text{Greedy}(\nu(P))$ and $N$ by
$\text{Greedy}(\nu(N))$. If the operand $(P,N)$ was
a minimal SABP of $n$, then clearly the new SABP also will be a minimal
SABP of $n$.

Operation 2:  For $(P,N)$ satisfying (A) and (B) and with
 $\max(P) = d+1+c$, for some $c > 0$ we define the following
operation: (Note that by the claim proved above, $d+c \in N$.) \\

\noindent
$P' \leftarrow (P \setminus \{d+c+1\} ) \cup \{ d+c \}$ \\
$N' \leftarrow (N \setminus \{ d+c, \min(N) \}) \cup \{ \min(N)-1,\min(N)-1 \}$ \\

It is easy to check that $\nu(P)-\nu(N) = \nu(P')-\nu(N')$ and that
$|P| = |P'|$. In the transformation for $N'$ the 0s are deleted if
$\min(N) = 1$, but we still have $|N'| \le |N|$.  Clearly if the operand
$(P,N)$ was  a minimal SABP of $n$, then
 $(P',N')$ also will be a minimal SABP of $n$.
We replace $(P,N)$ with $(P',N')$.

The transformation of a minimal SABP $(P,N)$ to a normal SABP is achieved
by the following procedure:  Since $\nu(P) \ge n$, 
if $(P,N)$ satisfies (B), and if   $\max (P) \notin \{ d, d+1 \}$,
then $\max (P) = d+ 1 + c$ for some $c > 0$, by Lemma \ref {lemma:greedseq}.

Step 1:   Apply operation 1 on $(P,N)$. If $\max (P) \in \{ d, d+1 \}$,  then
   stop
  and output $(P,N)$.

Step 2.  Apply operation 2
   on $(P,N)$ and go to step 1.

Note that for a minimal SABP, property (A) is trivially valid. It is
easy to verify that operation 2 can be applied on $(P,N)$ in step 2.
After each execution of step 1 and step 2, $(P,N)$ remains to be
a minimal SABP of $n$. Note that each time step 2 is executed,
 $\nu(P)$ reduces  by  $2^{d+c} > 2^d$. Since in
any minimal SABP $(P,N)$ of $n$, $\nu (P) \ge  n $, the procudure should
end after a finite number of steps. When the procedure ends, $(P,N)$
clearly satisfies properties (B) and (C).

\end{proof}

Note that condition (C) is not redundant. Although it is always true
that (when (B) is satisfied) $\max(P) \ge \lfloor \lg n \rfloor$, for a minimal SABP it is
not always the case that $\max(P) \le 1+ \lfloor \lg n \rfloor$. For
example, $5 = 15-7-3$ is a minimal SABP.

\section{Isoperimetric problems on $\mathcal{T}_\infty$}

\subsection{Ralation with Tanny and Conolly Sequences}

The first glimpse of the relationship between meta-Fibonacci sequences and
the discrete isoperimetric problem appreared in a
paper by Bharadwaj, Chandran and Das  \cite{BCD}, where they
related Tanny sequence with the connected edge isoperimetric sequence
of the infinite binary tree with all leaves at
the same level  ${\mathcal T}_{\infty}$.  
Though an independent proof was presented there, the result can also be obtained using
the combinatorial interpretation of Tanny sequences developed earlier
by Jackson and Ruskey \cite {JacksonRuskey}.  For a induced forest $F$ of ${\mathcal T}_{\infty}$,
we use $L(F)$ to denote the number of leaves of $F$ at the lowest level
of ${\mathcal T}_{\infty}$.

\begin{theorem} % [\cite {BCD}]
For all $n \ge 1$,
\begin{equation}
\delta_C(n) = n + 2 - 2 T(n).
\label{eq:CT}
\end{equation}
\end{theorem}

\begin{proof}
Let $S$ be a subtree of size $n$ of $\mathcal{T}_\infty$.
If $v$ is a vertex in a graph, then by $d(v)$ we denote the degree of $v$ 
in $\mathcal {T}_\infty$.
Note that

\[
\sum_{v \in S} d(v) = L(S) + 3(n-L(S)) = 3n - 2 L(S).
\]
On the other hand, because $S$ is a tree,
\[
\sum_{v \in S} d(v) = 2(n-1) + |(S,\overline{S})|.
\]
Observe that
\[
|(S,\overline{S})| = 3n - 2 L(S) - 2n + 2 = n + 2 - 2 L(S).
\]

Thus any subtree $S$ that maximizes $L(S)$
will be such that $|(S,\overline S)| = \delta_C(n)$.
In Jackson and Ruskey
\cite{JacksonRuskey} it is shown that
 $ T(n) = \max_{|S|=n}  L(S)$, where  $S$ is a subtree of
 ${\mathcal T}_{\infty}$.
\end{proof}

\begin{comment}
If we carry through the calculation we did above on a forest of $k$
trees, call it $F$, then we obtain
\[
\delta_C(F) = n + 2k - 2n_1.
\]
\end{comment}

\begin{figure}
\includegraphics[width=5in]{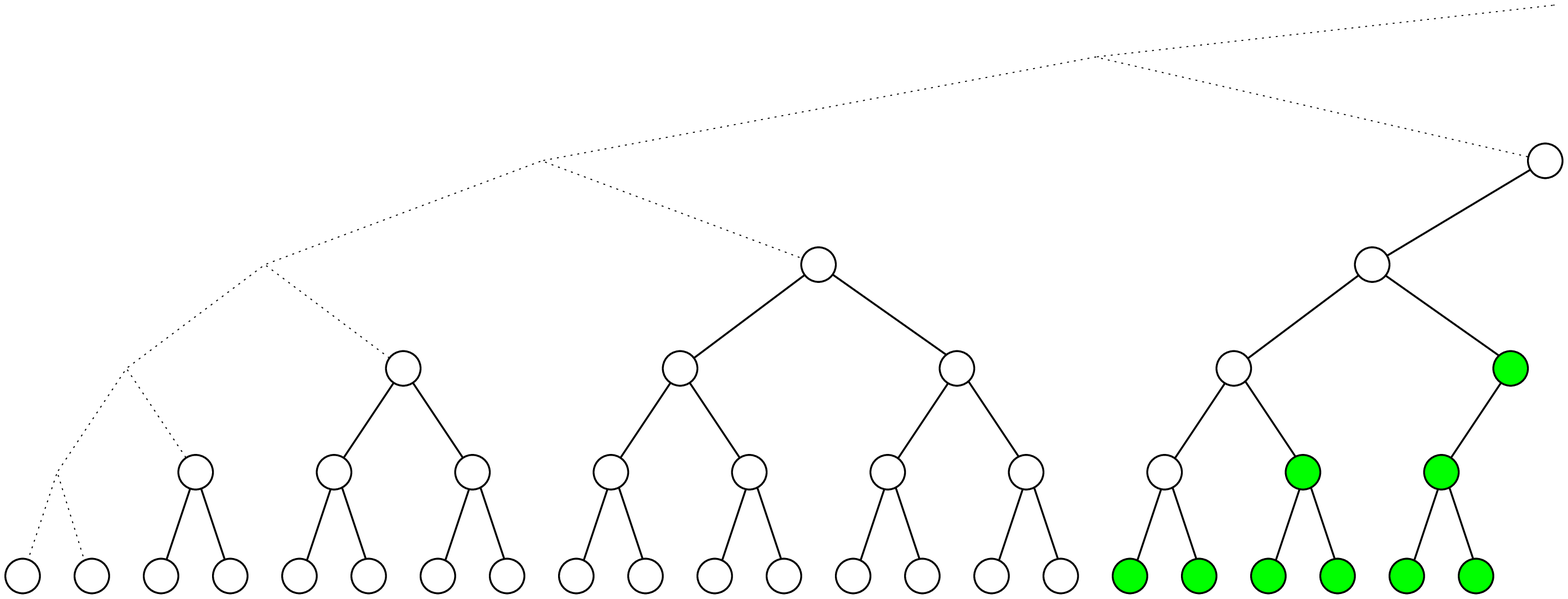}
\caption{The tree/forest $\mathcal{F}_0(40)$, showing the substructure
  $\mathcal{F}_0(9)$ as darkened nodes.}
\label{fig:treeC40}
\end{figure}

Our next aim is to get a similar relation between Conolly number $C(n)$
and the co-connected edge isoperimetric number $\delta_P(n)$. To do this
it is essential to establish that $\delta_P(n) = \tau_P(n)$.

\begin {definition}
{\bf The P-forest of  an ABP $(P, \emptyset)$:}  Let $(P, \emptyset)$ be
an ABP of $n$.  We
define the {\it P-forest} $F$  of $(P, \emptyset)$ to be  a forest induced in
 $\mathcal {T}_\infty$ as the disjoint union of $|P|$ complete binary trees,
such that for each $t$ in the multi-set $P$ we have a
tree of size $2^t - 1$ in the forest $F$  with its root at height $t$
from the leaf level, and having all their leaves at
the lowest level of $\mathcal {T}_\infty$. 
 Thus if $F$ is the  P-Forest of $(P,\emptyset)$,
$|F| = n$, $\overline F$ is connected in ${\mathcal T}_{\infty}$
 and $|(F,\overline F)| = |P|$.
\end {definition}

\begin{lemma}
\label {deltaPtauPlemma}
\begin{equation}
\delta_P(n) = \tau_P(n).
\end{equation}
\end{lemma}

\begin{proof}
Clearly every ABP $(P,\emptyset)$ of $n$ has  a P-forest $F
\subset \mathcal{T}_\infty$ such that $|F|
= n$, $\overline F$ inducing a connected
subgraph in ${\mathcal T}_{\infty}$
 and $|(F,\overline{F})| = |P|$. It follows that
$\delta_P(n) \le \tau_P(n)$.  Conversely,
any subset of vertices with $|S| = n$ and $\overline S$ connected
in ${\mathcal T}_{\infty}$, is such that   $S$
comprises of  a collection of  complete binary trees with all leaves
at the lowest  level  in $\mathcal{T}_\infty$. Such a subset $S$
 can be mapped into an ABP
$(P,\emptyset)$ by mapping each complete tree of size $\nu_j$ to an
integer $j \in P$; with the result that $|S| = n = \nu(P)$ and
$|(S,\overline{S})| = |P|$. Thus $\tau_P(n) \le \delta_P(n)$. The
Lemma follows.
\end{proof}

We denote by $L_P(n) = L(F)$ where $F$ is the P-forest of the ABP
$(\text Greedy(n), \emptyset)$. We will first prove $L_P(n) = C(n)$.
For this we need a result from  \cite{JacksonRuskey}, to state which
we need the following notions.

Let ${\mathcal F}_{\infty}$ be the infinite forest consisting of the
infinite sequence of complete binary trees $B_0, B_1, B_2, \ldots$, where
for $i \ge 1$, $B_i$ is the complete binary tree of depth $i$, and $B_0$ is
the single vertex tree.  (Depth of a complete binary tree is the number
of nodes in the path from the root to one of its leaves. Note that for $i \ge 1$,  $B_i$ contains
$\nu_i$ vertices. Thus $B_1$ is also a single vertex tree.)
Note that ${\mathcal F}_{\infty}$ can be seen as
an induced forest of ${\mathcal T}_{\infty}$. It is obtained when we
remove the (infinite) path from the parent of the first leaf
of ${\mathcal T}_{\infty}$ to the root of ${\mathcal T}_{\infty}$.
 (See Figure  \ref{fig:treeC40}: What should be removed from
$\mathcal {T}_\infty$ to get ${\mathcal F}_\infty$ is shown using dotted  lines.) In the rest of this sectin, when we
mention $\mathcal {F}_{\infty}$ we would be refering to this
induced forest of $\mathcal {T}_{\infty}$.  Also  the complete binary
tree $B_i$ will always refer to some induced complete binary tree of
depth $i$ in $\mathcal {T}_{\infty}$, with its root at the $i$th level
of $\mathcal {T}_{\infty}$ and all its leaves at the lowest level of $\mathcal {T}_\infty$.

 The vertices of ${\mathcal F}_{\infty}$
are numbered as follows: If $u \in B_i$ and $v \in B_j$ with $i < j$,
then $u$ is given a smaller number than $v$. The vertices within $B_i$
are numbered in the pre-order i.e.,  each vertex in $B_i$ is given  a smaller
number than the number given to any of its descendant and the left
subtree is numbered before the right subtree. 
We denote by ${\mathcal F} (n)$, the subforest of ${\mathcal F}_{\infty}$
induced by the first $n$ vertices with respect  to this numbering.
The following result is from
 \cite{JacksonRuskey}.

\begin {lemma}
[\cite{JacksonRuskey}]  $L({\mathcal F} (n) ) = C(n)$.
 \label {lemma:conollyleaf}
\end {lemma}

A pre-order prefix of $B_k$ having $x$ nodes, denoted
as $PP(x, B_k)$ is defined as the sub-tree of $B_k$ formed by the first
$x$ nodes visited when a pre-order traversal of $B_k$ starting from
the root is done.
It is easy to verify that for $x' \leq x$, $PP(x',B_k)$ is contained in
$PP(x, B_k)$. Note that $PP(k-1, B_k)$  is the path from the root
of $B_k$ to the parent of the left most leaf of $B_k$.
This path is called the {\it primary path} of $B_k$. The
following lemma is easy to verify.

\begin{lemma}\label{lem:helpinglemma}
Let  $x \geq k-1$. Then 
$L(PP(x,B_k)) =  L(\mathcal{F}(x-(k-1)))$.
\end{lemma}

\begin {proof}
Let $F'$ is the forest obtained by removing
$PP(k-1, B_k)$ from $PP(x, B_k)$. Then 
$L(PP(x,B_k)) = L(F') = L(\mathcal{F}(x-(k-1)))$.
\end {proof}

\begin{lemma}
For all $n \ge 1$,
\[
L_P(n) =  L( \mathcal {F} (n) ) .
\]
\label{lem:GreedConolly}
\end{lemma}

\begin{proof}
The proof is by induction on $n$.
For  $n=1,2$ etc,  it is easy to check that the Lemma holds.
Suppose that $L_P(j) = L(\mathcal {F} (j))$ for all $j < n$.

Clearly, there exists a unique positive integer $k$ such that,
$n =  1+ \sum_{1 \le i \le k-1} \nu_i  + x$, where $0 \le x < \nu_k$.
Let $B^t = \bigcup_{0 \le i \le t} B_i$. Clearly $ \mathcal {F} (n)
= B^{k-1} \cup PP(x, B_k)$.
Note that $n = 2^k - k + x$. We consider two cases based on how
$x$ compares with $k-1$.

\noindent{\bf Case I.} When $x < k - 1$.

Since $\nu_{k-1}  \le n = 2^k - k + x < 2^k - 1 = \nu_k$,
 the greedy algorithm will first
select $\nu_{k-1}$, and thus the corresponding P-forest will contain the
complete binary tree $B_{k-1}$.
Therefore we get the following:
\begin{eqnarray}
\label {Pforestequation}
L_P(n)
& = & L(B_{k-1})  + L_P(  n- \nu_{k-1}  )
\end{eqnarray}

On the other hand since $\mathcal{F}(n) = B^{k-1} \cup PP(x, B_k) =
B_{k-1} \cup B^{k-2} \cup PP(x,B_k)$, we have
$L(\mathcal {F} (n)) = L(B_{k-1}) + L( B^{k-2} ) + L( PP(x, B_k) )$.
 But note that $L(PP(x, B_k) = 0 = L(PP(x, B_{k-1})$ since $x < k-1$.
Thus $  L(\mathcal {F} (n) ) = L(B_{k-1}) + L( B^{k-2} ) + L( PP(x, B_{k-1}) )
 = L (B_{k-1}) + L( \mathcal {F} (n - \nu_{k-1} ) )$ since
$\mathcal {F} (n - \nu_{k-1} ) = B^{k-2} \cup PP(x, B_{k-1} )$.
Now, by induction hypothesis we have  $L( \mathcal {F} (n - \nu_{k-1} ) )
 = L_p ( n - \nu_{k-1} )$.
 It follows from Equation \ref {Pforestequation}
  that $L_p(n) = L( \mathcal {F} (n)$.

\noindent{\bf Case II.} When  $x \geq k - 1$.

Since $\nu_{k+1} >  n = 2^k-k+x \ge 2^k-1 = \nu_k$,  the
greedy algorithm picks up $\nu_k$ first, and thus the corresponding
P-forest contains $B_k$. Therefore,

\begin{eqnarray}
\label {Pforesteequation2}
L_P(n)
& = & L(B_k) + L_P(x-(k-1))
 =  2^{k-1} + L_P( x-(k-1) )
\end{eqnarray}

On the other hand $L(\mathcal {F} (n)) = L(B^{k-1}) + L(PP(x, B_k)$.
We note that $L(B^{k-1} ) = 1 + \sum_{1 \le i \le k-1} 2^{i-1} = 2^{k-1}$.
Recalling that by Lemma \ref {lem:helpinglemma}, $L(PP(x,B_k)) = L(\mathcal {F} (x - (k-1) )$,
we get $L (\mathcal {F} (n) ) = 2^{k-1} +   L(\mathcal {F} (x - (k-1) ))$.
By induction hypothesis we have
$L(\mathcal {F} (x - (k-1) ) = L_P (x - (k-1))$. It follows
from Equation \ref {Pforesteequation2}  that $L_P(n)
=  L(\mathcal {F} (n) )$.

\end{proof}

\begin {corollary}
\label {cor:leafcorollary}
  $$ L_P (n) = C(n)$$

\end {corollary}

\begin {proof}
Follows from Lemma \ref {lem:GreedConolly} and Lemma \ref  {lemma:conollyleaf}.
\end {proof}

\begin{theorem}
\begin{equation}
\delta_P(n) = 2 C(n) - n.
\label{eq:GC}
\end{equation}
\end{theorem}

\begin{proof}
In  $\mathcal {T}_{\infty}$,  every node is either a leaf or has
two children. Let $S$ be a subset of vertices of
$\mathcal {T}_{\infty}$ inducing   a  P-forest corresponding
to a minimal ABP of $n$. Clearly $\delta_P(n) = c$, the number
of trees in the forest induced by $S$. Clearly
$L(S) + 3(n-L(S)) = \sum_{v \in S} d(v) = 2(n-c) + c= 2n- \delta_P(n)$.
From this, it is easy to
see that $ L_p(n) = L(S) = (n+\delta_P(n))/2$.
Using Corollary \ref{cor:leafcorollary}, we obtain $\delta_P(n) = 2
C(n) - n$, as desired.
\end{proof}

The following Theorem  was  conjectured to be true by Jeorg
Arndt \cite{Arndt} (see OEIS A100661).

\begin{theorem}
The generating function of $\delta_P(z)$ is
\[
\frac{z}{1-z} \left( 2 \prod_{k \ge 1}(1+z^{2^k-1}) - (1-z) \right)
\]
\end{theorem}
\begin{proof}
This follows from (\ref{eq:GC}) and the known generating function
(\ref{eq:gfTC}) for $C(n)$.
\end{proof}

{
\subsection{Relation with SABP, ABP and CABP}

\label {deltaTausection}
}

In this section we show that  $\delta(n) = \tau(n)$ and $\delta_C(n)
= \tau_C(n)$, among other things.

\subsubsection{To prove $\delta(n) \ge \tau(n)$}

Let $S$ be a set of vertices of $\mathcal{T}_\infty$, with $|S|=n$.
We will show that $|(S,\overline{S})|$ can be expressed as the number of
  parts in a SABP of $n$.
Define a function $f_S : V(\mathcal{T}_\infty) \rightarrow \mathbb{N}$ as follows.
Let $\ell(v)$ denote the level number of $v$ in $\mathcal {T}_{\infty}$. If $v$ is
a leaf of $\mathcal {T}_\infty$, we take $\ell (v) = 1$.
\[
f_S(v) = \begin{cases}
\nu(\ell(v)) & \text{ if } v \in S \text{ and } \mathrm{par}(v) \not\in S, \\
-\nu(\ell(v)) & \text{ if } v \notin S \text{ and } \mathrm{par}(v) \in S, \\
0 & \text{ otherwise.}
\end{cases}
\]

\noindent (See figure \ref {fig:BT3}, where we have illustrated the function
$f_S$ for a subset $S$ with $|S|=24$.)

\begin{theorem}
For any subset $S$ of $V(\mathcal{T}_\infty)$, with $|S| = n$, we have:
\[
n = \sum_{v \in V(\mathcal{T}_\infty)} f_S(v) \;\;\; \mbox{ and } \;\;\;
\]
\begin{equation}
|(S,\overline{S})| = |\{ v \in V(\mathcal{T}_\infty) : f_S(v) \neq 0 \} |.
\label{eq:cutval}
\end{equation}
If $S$ is connected, then there is exactly one positive term in
(\ref{eq:cutval}). \label{thm:inout}
\end{theorem}

\begin{proof}
The second equality is true because $f_S(v) \neq 0$ precisely when
$(v,\mathrm{par}(v))$ is an edge of the cut $(S,\overline{S})$.

To prove the first equality think of labeling each node of
$\mathcal{T}_\infty$ by a multiset of $+1$s and $-1$s. If $f(v) =
\nu(\ell(v))$ then add a label $+1$ to each of the $\nu(\ell(v))$
nodes in the subtree rooted at $v$. If $f(v) = -\nu(\ell(v))$ then
add a label $-1$ to each of the $\nu(\ell(v))$ nodes in the subtree
rooted at $v$. Clearly the sum of the labels in each multiset,
summed over all the nodes in $\mathcal{T}_\infty$, is equal to
$\sum_{v \in V(\mathcal{T}_\infty)} f_S(v)$. However, we claim that
the sum of the labels at a node $v$ is $+1$ if $v \in S$ and is $0$
if $v \not\in S$. To see this, consider the (infinite) path that
starts at $v$ and  then successively contains each ancestor of
$v$.

If $v \in S$ then the path will contain subpaths of nodes that are
in $S$, then not in $S$, and so on, alternately, until reaching the
infinite subpath of nodes not in $S$. Each time that a subpath
changes status, a $+1$ or a $-1$ was added to the labels of $v$.
Since the number of such changes is odd, and the first change
corresponds to a $+1$, the total sum is $+1$.

If $v \not\in S$, then a similar argument shows that the total sum of the labels is 0.
Thus the sums of the labels over all nodes is equal to $n$.

If $S$ is connected, then since it must be a tree, there is only one
node $v$ such that $v \in S$ and $\text{par}(v) \not\in S$. Thus
there is only one positive term in (\ref{eq:cutval}).
\end{proof}

\begin{corollary}
\begin{equation}
\delta(n) \ge \tau(n) \;\;\; \mbox{ and } \;\;\; \delta_C(n) \ge \tau_C(n).
\label{eq:ineq}
\end{equation}
\label{corr:ge}
\end{corollary}

\begin{proof}
By Theorem \ref{thm:inout}, every $S \subset \mathcal{T}_\infty$ can
be mapped to a SABP $(P,N)$ such that $|S| = \nu(P) - \nu(N)$ and
$|(S,\overline{S})| = |P|+|N|$, where $P= \{ \ell(v) : f_S(v) 
 \mbox { is positive } \}$,
and $N = \{ \ell(v): f_S(v) \mbox { is negative }  \}$. Moreover if
$S$ is connected then by Theorem \ref {thm:inout}, $|P| = 1$, i.e.
$(P,N)$ is a CABP. From this the second part of the Theorem follows.
\end{proof}

\subsubsection{To prove $\delta(n) \le \tau(n)$}

We now show that the inequality of (\ref{eq:ineq}) is in fact an
equality.

Just like we define a P-forest corresponding to an ABP $(P, \emptyset)$ of
a positive integer $n$, now we will define a tree (more precisely a 
subtree of  $\mathcal {T}_{\infty}$) that corresponds to a CABP
$( \{r\}, N)$ of a positive integer $n$. (We will assume that
$N = Greedy(\nu(N))$, and therefore by part (1) of Lemma \ref {lemma:greedseq},
only the smallest number in $N$ can possibly repeat. If it repeats,
it repeats only twice.) 
 We define the {\it
C-Tree} of $( \{r\}, N)$  as follows: consider a subtree of  
$\mathcal {T}_{\infty}$
with its root, say $v_r$, at a height $r$. Now
define a path ($v_r,v_{r-1},\ldots,v_h$) starting from $v_r$ as
follows: $v_{j-1}$ is defined to be the right child of $v_j$ if and only if  $j-1 \notin N$, else it is defined to be the left child of $v_j$, for $r
\ge j \ge  h+1$. If $N$ does not have any repeated members, then $h=1$,
else $h=t+1$, where $t = \min N$, the repeated (smallest) element in $N$. Now
construct the C-Tree of $( \{ r \}, N)$ from the subtree rooted at $v_r$ by
the following procedure: For $j=r$ to $h+1$, prune away the subtree
rooted at the right child of $v_j$ whenever  $v_{j-1}$ is the left child of $v_j$. If  $j=h$ then if
$h \ne 1$ prune away the subtrees rooted at both its children.
It is easy to see that  the number of vertices in the tree $S$  constructed
using the above method is exactly $\nu_{r} - \sum_{i \in N} \nu_i = n$,
 and $|(S,\overline S)| = |N| + 1 = |(\{r\} , N)|$.

\begin {theorem}
\label {deltaCtauCtheorem}

 $\delta_C(n) = \tau_C(n)$.

\end {theorem}

\begin {proof}
Let $ (\{r\}, N)$ be  a minimal CABP of $n$ 
in normal form.  Let $S$ be the C-tree of $( \{r\}, N)$.
By the discussion above, $\delta_C(n) \le |(S,\overline S)| \le  |( \{r\}, N)| = \tau_C(n)$. The
Theorem follows, by combining with  Corollary \ref {corr:ge}
\end {proof}

\begin{theorem}\label{tau:n:is:delta:n:}
\[
 \delta(n) = \tau(n).
\]
Moreover, for any $n$ there exists a subforest $S$ of $\mathcal {T}_{\infty}$ such
that $|S| = n, |(S, \overline S)| = \tau(n)$ and  such that all the trees in the forest,
$S$, except possibly one are complete binary trees. If $(P,N)$ is a minimal
SABP of $n$ in normal form, the  subforest obtained by taking the disjoint union
of the C-tree of $(\{\max P\}, N)$ and the P-forest of
 $(P \setminus \{\max P\}, \emptyset)$ is such a subforest. 
\end{theorem}

\begin{proof}
In Corollary \ref{corr:ge} we proved that $\delta(n) \ge \tau(n)$.
Below we will show that $\delta(n) \le \tau(n)$.

Let $(P,N)$ be a minimal SABP of $n$. By Theorem \ref{thm:normal} we
can assume that $(P,N)$ is in normal form. We will show that there
is a set $S \subset V(\mathcal{T}_\infty)$ where $|S| = n$ and
$|(S,\overline{S})| = |N|+|P|$, and such that all the trees in $S$,
except possibly one are complete.

If $N = \emptyset$ then we simply
take disjoint complete binary trees with all
leaves at the lowest level of $\mathcal {T}_\infty$  of size $\nu_j$ for each $j \in
P$.
Otherwise, $\max(P) > \max(N)$.  Since by
part (3) of Lemma \ref {lemma:greedseq}, 
$\nu(N) \le \nu_{\max (N) + 1} - 1 <
\nu_{\max(P)}$, we infer that  $(\{ \max P\}, N)$ is the CABP of
some positive integer $n$.  Let $S$ be  the forest consisting
of the C-tree of $(\{ \max P \}, N)$ and the P-forest of the
ABP $(P- \{ \max P\}, \emptyset)$. Cleary $|S| = n$ and
$|(S, \overline S)| = |P| + |N|$. Moreover since all the trees in
a P-forest are complete binary trees with
all leaves at the lowest level of $\mathcal {T}_\infty$, 
$S$ can contain at most one tree which is
not complete.
\end{proof}

\begin{theorem}
\label {thm:isoperimetricvalue}

\begin{equation*}
\delta(n) = \min_{0 \le v \le n} \{ \delta_P(v) + \delta_C(n-v) \} =
n+2+ 2 \min_{0 \le v \le n} \{ C(v) - T(n-v) - v \}.
\end{equation*}
\end{theorem}

\begin{proof}
Clearly $\delta(n) \le \min_{0 \le v \le n} \{ \delta_P(v) +
\delta_C(n-v) \}$. By Theorem \ref {tau:n:is:delta:n:}, for any $n$, we can
find a subforest $S$ of $\mathcal {T}_{\infty}$
with $|(S,\overline S)| = \delta(n)$ such
that at most one of its trees is not a complete binary tree, with
all leaves at the lowest level of $\mathcal {T}_\infty$. 
Clearly these binary trees together form a P-forest $S'$  of  the ABP
of some number $v'$, where $0 \le v' \le n$. Also $S-S'$ is 
a connected subtree of $\mathcal {T}_\infty$. Therefore 
the number of out going edges from $S'$ is at least $\delta_P(v')$ and
the number of out going edges from $S-S'$ is at least $\delta_C(n-v')$. 
It follows that   $\delta(n) \ge \min_{0 \le v \le n} \{
\delta_P(v) + \delta_C(n-v) \}$. The second equality follows from
(\ref{eq:CT}) and (\ref{eq:GC}).
\end{proof}

\subsection{Towards a better understanding of $\delta(n)$}

Though Theorem \ref {thm:isoperimetricvalue} allows us to express $\delta(n)$
in terms of $\delta_P(n)$ and $\delta_C(n)$, 
it would be nice to have a better
understanding of the sequence $\delta(n)$.
When we study table \ref {table:numbers} containing
 values of $\delta(n),\delta_P(n)$ and
$\delta_C(n)$ for small values of $n$ we cannot fail to notice that
for a remarkably large number of columns in the table, the entry from
the third row equals either the entry in the first row or the second row.
That is  either $\delta(n) = \delta_P(n)$ or $\delta(n) = \delta_C(n)$.
This observation motivates us to carefully consider the two sets,
${\cal X} = \{ n \in \mathcal N \setminus \{ 0 \}:  \delta(n) = \delta_P(n) \}$ and 
${\cal Y} = \{ n \in \mathcal N \setminus \{ 0 \} : \delta(n) = \delta_C(n) \}$.  We would like
to carefully consider the question of characterising the numbers in
${\cal X}$ and ${\cal Y}$.  First we will show that the sets ${\cal X}$
and ${\cal Y}$ are intimately related with each other: There is a one to
one correspondence between these sets.  To prove this we need to note
some symmetries in the sequences correspoding to $\delta(n),\delta_P(n)$
and $\delta_C(n)$.  We explain this by defining the dual of a number:

\begin {definition}
{\bf The dual function:} Let the function 
 $f : \mathcal N \setminus \{ 0 \} \rightarrow \mathcal N \setminus \{ 0 \}$ be defined as follows:
If $n = \nu_k$, for some $k \ge 1$, then $f(n) = n$. Else, 
$f(n) = 3.2^d-n-2$, where $d = \lfloor \log n \rfloor$.  We say that
$f(n)$ is the dual of $n$. 
\end {definition}

\begin {lemma}
\label {dualofduallemma}

 The dual of the dual of $n$ equals $n$. That is $f(f(n))  = n$. 

\end {lemma}

\begin {proof}
If $n = \nu_k$, for some $k \ge 1$, then clearly  $f(f(n))  = n$. 
Else let $n' =f(n) = 3.2^d-n-2$, where $d = \lfloor \log n \rfloor$. 
Since $2^d \le n < 2^{d+1} - 1$, clearly $2^d \le n'  < 2^{d+1} - 1$ also.
Thus $ \lfloor \log {n'} \rfloor  = d = \lfloor \log n \rfloor $.
  Thus $f(n') = 3.2^d - n' - 2 = n$, as
required. 
\end {proof}

\begin {lemma}
\label {dualSABPlemma}
(1) $\tau (f(n)) \le \tau (n)$, \ \ \ (2) $\tau_P(f(n)) \le \tau_C(f(n))$ 
 \ \ \ (3) $\tau_C(f(n)) \le \tau_P(n)$.  
\end {lemma}

\begin {proof}

If $n = \nu_k$, for some $k \ge 1$, then 
clearly all the three statements are true,
since in this case $f(n) = \nu_k$ and  
therefore $\tau(f(n)) = \tau_P(f(n)) = \tau_C(f(n)) = 1$. 
%we can take $(P',N') = (\{k\}, \emptyset )$, which an SABP, CABP and ABP at
%the same time.   
Now let $n \ne \nu_k$, for $k \ge 1$.  Given a  minimal SABP (ABP or CABP) $(P,N)$
of $n$ in normal form, define $P'$ and $N'$ as follows.   

Let $N' = P \setminus \{ \max(P) \}$. 
Recall that by the definition of normal form, $\max P \in \{ d, d+1\}$ where
 $d = \lfloor \log n \rfloor$. Now define $P'$ as follows:

\begin{equation}
P' = \begin{cases}
 \{d+1\} \cup N  & \text{ if } \max(P) = d, \\
 \{d\}   \cup N  & \text{ if } \max(P) = d+1.
\end{cases}
\label{eq:Ddef}
\end{equation}

Note that, if $P$ is a multi-set and $\max(P)$ repeats in $P$ then
to get $N'$  only one copy of $\max P$ will be removed from $P$.
Similarly if $N$ already contains $d$,  $\{d\} \cup N$ will contain
one more copy of $d$.

It is easy to see that $|(P',N')| = |(P,N)|$. Let $(P',N')$ correspond
to $n'$.  Then  $n' = \sum_{j \in
{P'}} \nu(j) - \sum_{j \in {N'}} \nu(j)$. Therefore  $n + n' = \nu_{d+1} +
\nu_d = 3 \cdot 2^d -2$, so that $n' = 3 \cdot 2^d -2-n = f(n)$, as
required.   It follows that $\tau(f(n)) \le |(P',N')| = |(P,N)| = \tau(n)$.

Finally
if $(P,N)$ is a ABP then $N = \emptyset$ and thus $|P'|=1$ so that
$(P',N')$ is a CABP.  If
$(P,N)$ is a CABP then $|{P}|=1$ and thus $|{N'}|=0$ so that
$(P',N')$ is an ABP. From this we can infer that $\tau_P(f(n)) \le \tau_C(n)$
and $\tau_C(f(n)) \le \tau_P(n)$. 

\end {proof}

\begin{theorem}\label{lemma:dual for delta}
\[
\tau(n)   = \tau( f(n) ) \;\;\; \text{and} \;\;\;
\tau_C(n) = \tau_P( f(n)  ) \;\;\; \text{and}  \;\;\; \tau_P(n) = \tau_C(f(n)) .
\]
\end{theorem}

\begin{proof}

By Lemma \ref {dualSABPlemma}, we have $\tau(f(n)) \le \tau(n)$.
Recalling that by Lemma \ref {dualofduallemma}, we have
$f(f(n)) = n$, $ \tau(n) \le \tau(f(n)$ also, by applying Lemma
\ref {dualSABPlemma} to $f(n)$. The other equalities follow by a similar
argument.  
\end{proof}

%================================================================================================

\begin{theorem}
\label {dualDeltaTheorem}
For all $n \ge 2$,
\[
\delta_C(n) = \delta_P(f(n)) \mbox{ and }
\delta_P(n) = \delta_C(f(n))
\]
and
\[
\delta(n) = \delta(f(n)).
\]
\end{theorem}

\begin{proof}
This is immediate from  Lemma \ref {deltaPtauPlemma} 
  and Theorems \ref {deltaCtauCtheorem}, 
\ref {tau:n:is:delta:n:} and  \ref {lemma:dual for delta}.
\end{proof}

Now we are in a position to state the  relation between the 
two sets ${\cal X} =  \{ n \in \mathcal N \setminus \{ 0 \}: \delta(n) = \delta_P(n) \}$ and
${\cal Y }= \{ n \in  \mathcal N \setminus \{ 0 \} : \delta(n) = \delta_P(n) \} $.  Also
define ${\cal X}_n = \{ k \in {\cal X}:  k < n \}$, and
${\cal Y}_n  = \{ k \in {\cal Y}: k < n \}$.   

\begin {theorem}
\label {XYthm}
Let $n$ be a positive integer and let $f(n)$ be its dual. Then,
\begin {enumerate}
\item  $n \in {\cal X} $ if and only if   $f(n) \in {\cal Y}$.

\item $n \in {\cal Y}$ if and only if  $f(n) \in {\cal X}$. 

\item  $|{\cal X}_{2^{d+1}}| = |{\cal Y}_{2^{d+1}}|$
\end {enumerate}

\end {theorem}

\begin {proof}

If $n \in {\cal X}$, then $\delta (n)  = \delta_P (n)$. But
by Theorem \ref {dualDeltaTheorem}, we have $\delta(n) = \delta (f(n))$
and $\delta_P(n) = \delta_C(f(n))$. It follows that 
$\delta_C(f(n)) = \delta(f(n))$, i.e. $f(n) \in {\cal Y}$. The second
statement can be proved by a similar argument.  Finally note
that if $2^d \le n \le  2^{d+1}-1$ we also have $2^d \le f(n) \le 2^{d+1} -1$.
$f(2^d) = 2^{d+1} - 2, f(2^d+1) = 2^{d+1} - 3, \ldots, f(2^{d+1} - 2) = 2^d$ and so on, while
$f(2^{d+1} - 1) = 2^{d+1} - 1$. We infer from first and second statements
that $|{\cal X}_{2^{d+1}}| = |{\cal Y}_{2^{d+1}}|$.  

\end {proof}

The above Theorem implies that if  we can characterise the numbers
in  the set ${\cal X}$  we can also characterise the
number in the 
set ${\cal Y}$.  Now we discuss an algorithmic motivation for studying
the sets ${\cal X}$ and ${\cal Y}$.

\noindent {\bf Complexity of Computing  $\delta(n)$: A motivation
for studying the set ${\cal X}$ and ${\cal Y}$: }
How efficiently can we compute $\delta_P(n), \delta_C(n)$ and $\delta(n)$?
In view of Lemma \ref {lemma:greedy}, we know that
\begin {eqnarray}
\label {deltaPreclemma}
\delta_P(n)  &=& 1 \mbox { if $n  = \nu_k$ for some $k \ge 1$ } \\
    &=& 1 + \delta_P( n - \nu_{\lfloor \log n \rfloor}) 
\end {eqnarray} 
Therefore we can compute $\delta_P(n)$ in $O(\log n)$ time. Now
using Theorem  \ref {dualDeltaTheorem},  we know that $\delta_C(n) 
= \delta_P(f(n))$,
and thus $\delta_C(n)$ also can be computed in $O(\log n)$ time, recalling
that $\lfloor \log f(n) \rfloor = \lfloor \log n \rfloor$. 
To compute $\delta(n)$ we can use Theorem \ref {thm:isoperimetricvalue}:  
Let us use two
arrays of size $n'=2^{\lceil \log n \rceil}$ each,  
to store the values of  $\delta_P(k)$ and
$\delta_C(k)$ respectively  for
$1 \le k \le n'$. It is easy to see that this can be done in $O(n)$ time,
using Equation \ref {deltaPreclemma} and  then Theorem \ref {dualDeltaTheorem}.
Now we can compute $\delta(n)$ in  $O(n)$ time using 
Theorem \ref {thm:isoperimetricvalue}.

Can we compute $\delta(n)$ in $o(n)$ time ? As of now, we do not know
any algorithm for this. But we observe that
 Theorem \ref {thm:isoperimetricvalue}
can be rewritten as 

\begin {eqnarray}
 \delta(n) = \min_{v \in {\cal X }, n-v \in {\cal Y} }
  \delta_P(v) + \delta_C(n-v) 
\end {eqnarray}
  To see this note that  if 
 $\delta(v) <  \delta_P(v)$
then we have a subforest of ${\mathcal T}_\infty$ on $v$ vertices,
with number of out going edges strictly less than  $\delta_P(v)$. Now
taking the disjoint union of this subforest with 
a   subtree of $\mathcal {T}_\infty$ on
$n-v$ vertices with exactly $\delta_C(n-v)$ outgoing edges, 
we get a subforest of $\mathcal {T}_\infty$
with $< \delta_P(v) + \delta_C(n-v)$ out going edges. Thus,
$\delta(n) < \delta_P(v) + \delta_C(n-v)$. 
We infer that if $\delta(n) = \delta_P(v) + \delta_C(n-v)$, then 
$v \in {\cal X}$.  
A similar reasoning tells 
us that  $n - v \in {\cal Y}$.  Suppose we can enumerate
the members of ${\cal X}_n$ in ascending order in $O(|{\cal X}_n|)$ time. 
 Note that 
if $k \in {\cal X}$ and
$k \ne \nu_i$ for any $i \ge 1$, 
 then $k - \nu_{\lfloor \log k \rfloor } \in {\cal X}$
also. Thus using Equation \ref {deltaPreclemma},
we can store the members of
${\cal X}_{n'}$   (where $n' = 2^{\lceil \log n \rceil}$) 
along with the corresponding $\delta_P$ values in arrays,
just the same way we did earlier. Now that we have stored the members
${\cal X}_{n'}$ in arrays, we can store the members of ${\cal Y}_{n'}$ also
along with their corresponding $\delta_C$ values,  
by using Theorem \ref {XYthm}: For each member $k \in {\cal X}_{n'}$,
add $f(k)$  in 
${\cal Y}_{n'}$, and $\delta_C(f(k)) = \delta_P(k)$.    
From this  it
is easy to see that we can compute $\delta(n)$ in $O(|{\cal X}_n|)$ time,
provided we can generate the members of ${\cal X}_n$ 
in ascending order, in $O(|{\cal X}_n|)$ time. 
Based on the values for $|{\cal X}_n|$ for small values of $n$ we
conjecture that $|{\cal X}_n| = o(n)$, and leave open the question
of enumerating the members of 
${\cal X}_n$ in ascending order, in $O(|{\cal X}_n|)$ time.

As of now, we do not have a complete understanding of the set 
${\cal X}$. But  we will present a non-trivial sufficient condition 
(Theorem \ref {gap3theorem})for
a number $n$ to belong to ${\cal X}$, in terms of
the nature of the optimal ABP of $n$. In the last section we will
show an application of Theorem \ref {gap3theorem} 
 to improve the previously  known results
on the edge isoperimetric peak of complete binary trees.

Let $(P, \emptyset)$ be the ABP of a number $n$ where $P =
\{i_1,i_2,\ldots, i_h\}$ where $i_1 < i_2 < \ldots < i_h$. If for
each $j$ where $1 \le j < h$ we have $i_{j+1} - i_j \ge k$ for $k \ge 1$,
 we say
that the  ABP satisfies the ``gap-k condition''.
Note that if an ABP satisfies the gap-$k$ condition for
some $k \ge 1$, then it satisfies the
gap-$k'$ condition for all $ 1 \le k' \le k$.
  The following
observation is a direct consequence of Lemma \ref {lemma:greedseq} (1)
and Lemma \ref {lemma:greedy}.

\begin {observation}
\label  {gapobservation}

If an ABP of $n$ satisfies the gap-1 condition for some $k \ge 1$,
(i.e. if no terms repeat)  then
it is a greedy ABP and thus a  minimal ABP of $n$.

\end {observation}

\begin{theorem} \label{gap3theorem}
Let $n = \nu_{i_1} + \nu_{i_2} + \cdots + \nu_{i_k}$. If for every
$j$, $2 \leq j \leq k$, $i_j - i_{j-1} \ge 3$ $($i.e., if $n$
satisfies the gap-3 condition$)$ then we have $\delta (n) = \delta_P
(n)= k$.
\end{theorem}

\begin{proof}
In view of observation \ref {gapobservation},  we have
$\delta_P(n) = k$.
We prove that $\delta(n) = k$ by  induction on the number of terms $t$.
When $t = 1,2$, this is easy to verify.
Now let $t =k$ where $k \ge 3$. 
 Let us assume that the Theorem  is true for all
$t < k$. (If $k \ge 3$ then $i_k \ge 7$ because of
the gap-3 condition.) 

Suppose for contradiction that  $\delta(n) \le k-1$.

\noindent {\bf Claim 0:} Let $(P,N)$ be a minimal SABP of $n$. Then
$i_k \notin P$.

Suppose for contradiction that $i_k \in P$. Then consider the number
$n' = n - \nu_{i_k}$. Clearly $(P - \{i_k\}, N)$ is a SABP of $n'$.
Since we have assumed that $\delta(n) = \tau(n) \le k-1$, we get
$\tau(n') \le k-2$. This is a contradiction, since $n' = \sum_{j =
1}^{k-1} \nu_{i_j}$ satisfies the gap-condition, and thus by
induction hypothesis we should have $\tau(n') = k-1$.
$\Box$.

Consider any  minimal SABP  $(P,N)$ of $n$. By Theorem \ref
{thm:normalform}, we can assume that this minimum SABP is in normal
form. Let $\max(P)=i_m$ and $\max (N) = i_n$
respectively. Since $(P,N)$ is in normal form, we have $\max(N \cup
P) =   \max (P)= i_m \in \{ \lfloor \log n  \rfloor, \lfloor \log n
\rfloor + 1 \}$. Since $\nu_{i_k}  <  n < \nu_{i_k + 1} $ (by
Observation \ref {gapobservation} and Lemma \ref {lemma:greedseq} part (3)),
 it is easy to verify that  $\lfloor \log n \rfloor
= i_k$.
Thus $i_m \in \{i_k, i_k + 1\}$. In view of Claim 0, $i_m
\neq i_k$. Thus $i_m = i_k+1$.

\noindent {\bf Claim 1:} In any minimal SABP $(P,N)$ of $n$ in
normal form, $i_m$ does not repeat in the multiset $P$.

Suppose it repeats. Then since the SABP is assumed to be in the
normal form, $P = \text Greedy(\nu (P) ) $.
 Thus by Lemma \ref {lemma:greedseq},
 if $i_m$ repeats in the multiset $P$, $P =
\{i_m,i_m\}$. Clearly  $i_n < i_m$. Recalling that
$i_m = i_k + 1$,  we have $n = 2 \nu_{i_k + 1}
- \sum_{j\in N} \nu_{j} \ge (2^{i_k + 2} - 2) - (2^{i_k + 1} - 2) =
2^{i_k + 1} > \sum_{j=1}^k \nu_{i_j} = n$, a contradiction.
$\Box$

\noindent {\bf Claim 2:}  If $(P,N)$ is a minimal SABP of $n$
in normal form with $i_n = \max N$, then we have $i_n < i_k$.

Recall that $i_m = i_k+1$.
Since SABP $(P,N)$ is in normal form, $i_n < i_m$.
 Suppose $i_n = i_k$. Then we can get
another minimal SABP (not necessarily in
normal form), say $(P',N')$, for $n$, by taking $P' = P -
\{i_k + 1\} \cup \{i_k,1\}$ and $N' = N - \{ i_k \}$.
This is clearly a contradiction in
view of Claim 0 since $(P',N')$ is minimal, but  $i_k \in P'$.

Now consider a minimal SABP $(P,N)$ of $n$ in normal form.
By Theorem \ref {tau:n:is:delta:n:}, we can find a forest $S$ in $\mathcal {T}_{\infty}$,
such that $S= S' \cup S''$ where $S'$ is the C-tree of $(\{i_m\}, N)$
and $S''$ is the P-forest of $(P-\{i_m\}, \emptyset)$.

By  the definition of a C-tree, the tree $S'$ has height $i_m = i_k+1$.
We say that a node in  $S'$ (seen as a subtree of
$ \mathcal {T}_{\infty}$)  is {\it saturated} {\bf either} if it is a leaf of
$ \mathcal {T}_{\infty}$ {\bf or} if both its  children (with respect to
$\mathcal {T}_{\infty}$) belong to $S'$. Note that by the
definition of C-tree, in $S'$ a node at
a height $t$ is unsaturated if and only if $t-1 \in N$.
Let $r$ be the root of $S'$.
Since by
Claim 2, $i_k \notin N$, and since $r$ is at height $i_k + 1$,
we have the following claim.

\noindent {\bf Claim 3:}  The root $r$  of $S'$ is
saturated.

Let $x$ and $y$ be the right child and
left child of $r$, respectively. Note that by the defintion of
C-tree,  the subtree of $S'$ rooted at $y$ is
complete and  has $\nu_{i_k}$ vertices in it. Let $S_1$ represent
the tree obtained by removing the subtree rooted at $y$ from
$S'$. Then clearly, $S_1 \cup S''$ together is a forest in $\mathcal {T}_{\infty}$, on  $n' = n - \nu_{i_k}$ vertices, and with number of out going
edges equal to  $\delta(n) + 1 \le k-1 + 1 = k$. (We have to add $1$ to
$\delta(n)$ because a  new out going edge incident on $r$ is created
by the removal of the subtree rooted at $y$, namely the edge
$(r,y)$.) By induction hypothesis we know that $\delta(n') = k-1$,
since  $n'$ clearly has a ABP satisfying the gap-3 condition.
We will now show that by  a slight modification of
$S_1$,  we can reduce the number of out going edges by at
least $2$ and get a representation of $n'$ with only $k-2$ out going
edges which will be a contradiction to the induction hypothesis.
First we make an easy observation.

\noindent {\bf Claim 4:}  $n' < \nu_{i_k - 2}$.

Recalling that by gap-3 condition, $i_{k-1} \le i_k - 3$, we get:

\begin {eqnarray}
  n' &=& n - \nu_{i_k} = \sum_{j=1}^{k-1} \nu_{i_j} \\
     &\le & \nu_{(i_{k-1} + 1)} - 1 \\
     &<& \nu_{i_k - 2}
 \end {eqnarray}

$\Box$

Let  $x$  be the
right child of $r$.

\noindent {\bf Claim 5:}  $x$ is unsaturated in $S'$, but it
has a left child (say $x'$).

Since $r$ is at a height  $ i_k + 1$, $x$
is at a height of $i_k \ge 7$.
If $x$ is saturated it has a complete left subtree
with $\nu_{i_k - 1}$ vertices in it. Therefore $n' \ge
\nu_{i_k -1} > \nu_{i_k -2}$ which contradicts Claim 4. Thus
$x$ is unsaturated, i.e. it does not have a right
child. If $x$ does not have a left child also, $S_1$
contains only $2$ nodes, namely $r$ and $x$ but has
4 outgoing edges. Clearly this is not optimum for $2$
nodes: We can replace $S_1$ with two leaves of $\mathcal {T}_{\infty}$,
thereby reducing the total number of outgoing edges by 2, which
contradicts the induction hypothesis that $\delta(n') = k - 1$.
We infer that $x$ has a left child, say $x'$.

\noindent {\bf Claim 6:}  $x'$ is unsaturated in $S'$, but it
has a left child (say $x''$).

Clearly $x'$ is at a height of $i_k - 1$ and if it is
saturated it will have a complete left subtree  and
therefore we get  $ n' > \nu_{i_k -2}$ a contradiction
to Claim 4. Thus $x'$ has no right child.
Now if there is no left child also
for $x'$, $S_1$ contains only $3$ vertices, namely $r,x,x'$ and
together they have $5$ out going edges. This is clearly not the
optimum representation for $3$ vertices. Rather, there exists
representation for $3$  vertices with just one out going edge.

In view of Claim 5,
clearly there are $4$ out going edges incident on the vertices $r,x$
and $x'$. We replace  $S_1$ with a forest
consisting of the subtree of $S_1$ rooted at $x''$
and a complete binary tree of $3$ vertices reducing the number of out going
edges by $2$. Thus we get a representation for $n'$ using at most
$k-2$ out going edges, a contradiction to the induction hypothesis.
Hence the theorem.
\end {proof}

In view of the above theorem it is natural to ask if $n$ has
an ABP satisfying
the gap-2 condition rather than gap-3 condition, then can we still
say $\delta_P(n)=\delta(n)$. This is not true as the following example illustrates.

\begin{example}
Applying the greedy algorithm to $n = 46912496118419$, we
obtain that  $n = \nu({1,3,5,\ldots,45})$. This ABP   clearly satisfies the
gap-2 condition and shows that $\delta_P(n) = 23$.  On the other hand,
$n = \nu({46}) - \nu({7}) - \nu({8,10,12,\ldots,44})$, showing that
$\delta(n) \le 21 < 23 = \delta_P(n)$.
\end{example}

\subsection{Improved lower bound for edge isoperimetric peak for $B_d$ }

The {\it edge isoperimetric
peak} of  a finite graph 
$G$, denoted as $\hat{\delta}_G(n)$ where $|V(G)| = n$, is
defined as $\hat{\delta}_G(n) = \max_{1 \leq i \leq
n}{\delta(i,G)}$.

The problem of finding   the isoperimetric
peak of  a complete binary tree of depth $d$ (denoted as $B_d$)  was studied in
\cite{otachi} and \cite{BharadwajChandran}. In \cite{otachi} it is
shown that $\hat{\delta}_{B_d} \geq \frac{d-(8+2\log d)}{8+2\log d}$
and in \cite{BharadwajChandran} it is shown that $\hat{\delta}_{B_d}
\geq \frac{d}{5}$ (see the proof of Corollary 1 in
\cite{BharadwajChandran}).  We will show that
using Theorem \ref{gap3theorem} we can get  a better
lower bound for  the edge isoperimetric peak of $B_d$. 
To do this, we first make the following simple observation:

\begin {lemma}
\label {peaklowerboundconversionlemma} 
For  $1 \le n \le 2^d-1$,  $\delta (n, B_d) \ge \delta 
(n, \mathcal {T}_\infty) - 1$
\end {lemma}

Now we can get a better lower bound for the edge isoperimetric
peak of $B_d$, compared to the previous $d/5$. 

\begin {theorem}
\label {isoperimetricpeakthm}
 $\hat{\delta}_{B_d}
 \ge  \left \lfloor 
d/3 \right \rfloor - 1$.
\end {theorem}

\begin {proof}
Clearly if we take $n = \nu_1 + \nu_4 + \nu_7 + \ldots + \nu_{(\left \lfloor 
d/3 \right \rfloor - 1) 3 + 1} $, then  $n \le 2^d -1$.  By Theorem \ref {gap3theorem} and
Lemma \ref {peaklowerboundconversionlemma}, we get $\hat{\delta}_{B_d}
\ge \delta(n, B_d) \ge  \left \lfloor 
d/3 \right \rfloor - 1$.
$\Box$
\end {proof}

Note that in the context of the edge isoperimetric
peak problem, Theorem \ref {gap3theorem} gives us more than what is claimed in
Theorem \ref {isoperimetricpeakthm}. For any $k \le   \left \lfloor 
d/3 \right \rfloor - 1$, it allows to find some numbers $n < 2^d-1$, such
that $\delta(n, B_d) = k$.  The following Theorem captures this point.   

\begin {theorem}
If $k = \left \lfloor  d/3 \right \rfloor - 1 - t$, then 
$| \{ n :  n \le 2^d -1, \delta(n,B_d) \ge k \}|  \ge 
 {  {\left \lfloor  d/3 \right \rfloor } \choose t }$   
\end {theorem}

\begin {proof}
Consider the ABP  
$\nu_1 + \nu_4 + \nu_7 + \ldots + \nu_{(\left \lfloor 
d/3 \right \rfloor - 1) 3 + 1}$. 
We can remove any $t$ of the terms from this ABP to
get another ABP of $ \left \lfloor  d/3 \right \rfloor  - t$ terms,
and that ABP would clearly satisfy the gap-3 condition. By Theorem 
\ref {gap3theorem} each of these   ${  {\left \lfloor  d/3 \right \rfloor } \choose t }$ ABPs, corresponds to a distinct number $n < 2^d -1$, satisfying
the property $\delta (n,B_d) = k$. 
$\Box$
\end {proof}

\section{Acknowledgements}

The authors wish to thank Jeff Shallit, Jeorg Arndt, and David
Wasserman for useful input, particularly about $\delta_P(n)$.

\end{document}